\documentclass [11pt]{article}
\usepackage{amsmath,amsthm,amssymb}
\usepackage{epsfig,graphics}

\newcommand{\Ndash}{\nobreakdash--}

\newtheorem{theorem}{Theorem}
\newtheorem{lemma}{Lemma}[section]
\newtheorem{proposition}[lemma]{Proposition}
\newtheorem{corollary}[lemma]{Corollary}
\newtheorem{theo}[lemma]{Theorem}

\numberwithin{equation}{section}

\title{Oscillation of H\"older Continuous Functions}
\author{ Jos\'e Gonz\'alez Llorente and Artur Nicolau}
\date{}

\begin{document}

\maketitle

\footnotetext{Both authors are partially supported by the grants
MTM2011-24606 and 2009SGR420.}

\begin{abstract}
Local oscillation of a function satisfying a H\"older condition is
considered and it is proved  that its growth is governed by a
version of the Law of the Iterated Logarithm.

\end{abstract}

\section{Introduction}\label{section1}

For $0<\alpha<1$, let $\Lambda_{\alpha}(\mathbb{R})$ be the class of
functions $f\colon \mathbb{R}\to \mathbb{R}$ for which there exists
a constant~$C=C(f)>0$ such that $|f(x)-f(y)|\le C|x-y|^{\alpha}$ for
any $x,y\in\mathbb{R}$. The infimum of such constants~$C$ is denoted
by $\|f\|_{\alpha}$. For $b>1$, G. H. Hardy proved in \cite{H} that
the Weierstrass function
$$
f_{b}(x)=\sum^{\infty}_{j=1} b^{-j\alpha} \cos (b^{j}x), \quad x\in\mathbb{R},
$$
is in $\Lambda_{\alpha}(\mathbb{R})$ and exhibits the extreme
behaviour
$$
\limsup_{h\to 0} \frac{| f_{b}(x+h)-f_{b}(x)|}{|h|^{\alpha}}
>0
$$
for any $x\in \mathbb{R}$. However fixed $x\in\mathbb{R}$ one may
expect many changes of sign of $f_{b}(x+h)-f_{b}(x)$ as ~$h \to 0$.
Next definition provides a way of quantifying it. Given a
function~$f\in\Lambda_{\alpha} (\mathbb{R})$ and $0<\varepsilon <
1/2$, consider
\begin{equation}\label{eq1.1}
\Theta_{\varepsilon}  (f)(x)=\int^{1}_{\varepsilon}
\frac{f(x+h)-f(x-h)} {h^{\alpha}} \frac{dh}{h},\quad x\in\mathbb{R}.
\end{equation}
It is clear that $\|\Theta_{\varepsilon}(f)\|_{\infty} \le
2^{\alpha} \|f\|_{\alpha} \log ( 1/\varepsilon)$. Moreover this
uniform estimate can not be improved as the elementary example
$f(x)=|x|^{\alpha} \operatorname{sign}(x)$ shows. However at almost
every point~$x$, the uniform estimate can be substantially improved.
The main result of the paper is the following

\begin{theorem}\label{theo1}
Fix $0<\alpha<1$. For $f\in \Lambda_{\alpha}(\mathbb{R})$ and
$0<\varepsilon <1/2$, let  $\Theta_{\varepsilon} (f)(x)$ be given by
\eqref{eq1.1}. Then, there exists a constant~$c(\alpha)>0$,
independent of $\varepsilon$ and $f$, such that
\begin{enumerate}
\item[(a)] For any interval $I\subset \mathbb{R}$, $|I|=1$, one has
$$
\int_{I} |\Theta_{\varepsilon}(f)(x)|^{2}\,dx \le c(\alpha) (\log
1/\varepsilon) \|f\|^2_{\alpha}.
$$
\item[(b)] At almost every point~$x\in\mathbb{R}$, one has
$$
\limsup_{\varepsilon \to
0^{+}}\frac{|\Theta_{\varepsilon}(f)(x)|}{\sqrt{\log ( 1/\varepsilon
) \log \log\log ( 1/\varepsilon )}} \le c(\alpha) \|f\|_{\alpha}.
$$
\end{enumerate}
\end{theorem}

The main technical step in the proof is the following estimate which provides the right subgaussian decay: there exists a constant~$c=c(\alpha)>0$ such that for any $t>0$ one has
\begin{equation}\label{eq1.2}
|\{ x\in [0,1] : |\Theta^{*}_{\varepsilon} (f)(x)| >t \sqrt{\log (
1/\varepsilon )} \|f\|_{\alpha} \}| \le c e^{-t^{2}/c}.
\end{equation}
Here $\Theta^{*}_{\varepsilon}(f)$ is the maximal function given by
$\Theta^{*}_{\varepsilon}(f)(x)=\sup \{|\Theta_{\delta}(f)(x)|:
1/2\ge \delta\ge \varepsilon\}$. Theorem~\ref{theo1} follows from
this subgaussian estimate by standard arguments. Our proof of
\eqref{eq1.2} is organized in two steps. First we state and prove a
dyadic version of \eqref{eq1.2} and later we use an averaging
procedure due to J.~Garnett and P.~Jones~(\cite{GJ}).
Theorem~\ref{theo1} is sharp up to the value of the
constant~$c(\alpha)$. Moreover there exists~$f\in
\Lambda_{\alpha}(\mathbb{R})$ for which there exists a
constant~$c=c(f)>0$ such that for any $0<\varepsilon < 1/2$ one has
$$
\int^{1}_{\varepsilon}
\frac{|f(x+h)-f(x-h)|}{h^{\alpha}}\frac{dh}{h} >c\log (
1/\varepsilon )
$$
at almost every $x\in\mathbb{R}$. So, Theorem~\ref{theo1} holds due
to certain cancellations which occur in the integral defining
$\Theta_{\varepsilon}(f)(x)$.

Subgaussian estimates and Law's of the Iterated Logarithm play a
central role in the boundary behavior of martingales and have also
appeared in function theory. For instance, in the relation between
the boundary behaviour of a harmonic function in an upper half space
and the size of its area function (\cite{W}, \cite{CWW}, \cite{BM})
or in differentiability properties of functions defined in the
euclidean space~(\cite{AP}, \cite{SV}). Our result is inspired in
the nice work of Y.~Lyubarskii and E.~Malinnikova (\cite{LM}) who
studied the oscillation of harmonic functions in the Koremblum
class. Related results can be found in ~\cite{E}, ~\cite{EM} and
\cite{EMM}.

The paper is organized as follows. Section~\ref{section2} is devoted
to the dyadic version of Theorem~\ref{theo1}. The averaging
procedure which is used to prove the results in the continuous
setting from their dyadic counterparts, is given in
Section~\ref{section3}. Section~\ref{section4} contains the proof of
the subgaussian estimate~\eqref{eq1.2} as well as the proof of
Theorem~\ref{theo1}. In Section~\ref{section5}, the sharpness of the
results is discussed. Finally, Section~\ref{section6} provides a
higher dimensional analogue of Theorem~\ref{theo1}.

The letters~$c$ and $c(\alpha)$ will denote a constant and a constant depending on the parameter~$\alpha$ whose value may change from line to line.

It is a pleasure to thank Eugenia Malinnikova for several sharp
remarks on a first version of this paper.

\section{Dyadic Model}\label{section2}

For $1\le \rho \le 2$, let $\mathcal{D} =\mathcal{D}(\rho)$ be the
collection of intervals of the form~$[j2^{-k}\rho, (j+1)
2^{-k}\rho)$, where $j\in\mathbb{Z} $ and $k=0,1,2,\dotsc$ . Let
$\mathcal{D}_{k}=\mathcal{D}_{k}(\rho)$ be the collection of
intervals of $\mathcal{D}$ of length~$2^{-k}\rho$ and let
$\mathcal{F}_{k} = \mathcal{F}_{k} (\rho)$ be the $\sigma$-algebra
generated by the intervals of $\mathcal{D}_{k}$. In the rest of this
section the number~$1\le \rho\le 2$ is fixed. A dyadic martingale is
a sequence of functions~$\{S_{k}\}$ defined in $[0, \rho]$ such that
for any $k=0,1,2,\dotsc$ the following two conditions hold:
(a)~$S_{k}$ is adaptated to~$\mathcal{F}_{k}$; (b)~the conditional
expectation of $S_{k+1}$ respect to $\mathcal{F}_{k}$ is $S_{k}$. In
other words: $S_{k}$ is constant in each interval
of~$\mathcal{D}_{k}$ and
$$
\frac{1}{|I|} \int_{I} (S_{k+1}(x)-S_{k}(x))\,dx=0
$$
for any $I\in \mathcal{D}_k$, $k=0,1,2,\dotsc$. Given a dyadic
martingale~$\{S_{n}\}$, its quadratic variation~$\langle
S\rangle_{n}$ is defined as
$$
\langle S\rangle^{2}_{n} (x)=\sum^{n}_{k=1} (S_{k}(x)-S_{k-1}(x))^{2},\quad n=1,2,\dotsc .
$$
It is well known that the quadratic variation governs the boundary
behaviour of the martingale. More concretely, the sets $\{ x \in [0,
\rho] :  \lim\limits_{n\to\infty} S_{n}(x) \text{ exists} \}$ and
$\{x \in [0, \rho]: \langle S\rangle_{\infty}(x)<\infty \}$ coincide
except at most for a set of Lebesgue measure $0$. Moreover there
exits a universal constant $c>0$ such that
$$
\limsup_{n\to\infty} \frac{|S_{n}(x)|}{\sqrt{\langle
S\rangle^{2}_{n}(x) \log\log \langle S\rangle_n (x)}}\le c,
$$
at almost every point~$x$ where $\langle
S\rangle_{\infty}(x)=\infty$. We also mention that an elementary
orthogonality  argument gives that
$$
\int_{0}^{\rho} |S_{n}(x)|^{2}\,dx=\int^{\rho}_{0} \langle S\rangle^{2}_{n}(x)\,dx,\quad n=1,2,\dotsc .
$$

Fix $0<\beta <1$. Let $\{S_{n}\}$ be a dyadic martingale satisfying $\|S_{n}\|_{\infty}\le 2^{n\beta}$, $n=0,1,2,\dotsc$ . For $N=1,2,\dotsc$, consider
$$
\Gamma_{N}(x)=\Gamma_N(\{S_{n}\}) (x)=\sum^{N}_{k=1} 2^{-k\beta}
S_{k}(x).
$$
It is clear that $\|\Gamma_N\|_{\infty}\le N$. Moreover this uniform
estimate is best possible. Actually, if the initial
martingale~$\{S_{n}\}$ satisfies, $S_{0} \equiv 0$,
$\|S_{n}\|_{\infty}=2^{n\beta}$ and $S_{k}(x) =2^{k \beta}$ for some
$x\in\mathbb{R}$ and any $k \leq N$; then  $\|\Gamma_N
(\{S_{n}\})\|_{\infty}=N$. However, as next result shows, this
uniform estimate can be substantially improved at almost every
point. Parts~(b) and (c) are the discrete analogues of
Theorem~\ref{theo1}.

\begin{theorem}\label{theo2}
Fix $0<\beta<1$ and $C>0$. Let $\{S_{n}\}$ be a dyadic margingale
with respect $\mathcal{D}(\rho)$ with $S_{0}\equiv 0$ and
$\|S_{n}\|_{\infty}\le C 2^{n\beta}$, $n=1,2,\dotsc$ . For
$N=1,2,\dotsc$, consider
\begin{align*}
\Gamma_N(x)&= \sum^{N}_{k=1} 2^{-k\beta} S_{k}(x),\\*[3pt]
\Gamma_N^{*}(x)&= \sup_{k\le N} |\Gamma_{k}(x)|.
\end{align*}
Then, there exists a constant $c= c(\beta , C)>0$ such that
\begin{enumerate}
\item[(a)] For any $\lambda >0$ and any $N=1,2,\dotsc,$ one has
$$
\int^{\rho}_{0} \exp \left( \lambda \Gamma^{*}_{N}(x) \right) \,dx
\le c e^{c \lambda^{2}N}.
$$
\item[(b)] For any $N=1,2,\dotsc$, one has
$$
\int^{\rho}_{0} |\Gamma_N^{*}(x)|^{2}\,dx \le c N.
$$
\item[(c)] For almost every $x\in [0,\rho]$ one has
$$
\limsup_{n\to\infty} \frac{|\Gamma_N(x)|}{\sqrt{N \log\log N}}\le c
.
$$
\end{enumerate}
\end{theorem}

\begin{proof}
We can assume $C=1$. Although $\{\Gamma_N\}$ is not a dyadic
martingale, we will show that its size is comparable to the size of
a dyadic martingale with bounded differences. Actually, consider the
dyadic martingale~$\{T_{n}\}$ defined by $T_{0}\equiv 0$ and
$$
T_{n}=\sum^{n}_{k=1} \frac{S_{k}-S_{k-1}}{2^{k \beta}}, \quad
n=1,2,\dotsc\,.
$$
The subgaussian estimate (see \cite[p.~69]{BM}) gives that
$$ |\{x
\in [0, \rho ]: T_{n}^{*}(x)
> t \}| \leq 2 \exp(-t^2 / 2 \|\langle T\rangle^{2}_{n}\|_{\infty}
)\, ,$$ for any $t >0$. Here $T_{n}^{*} (x) = \sup \{ |T_k (x)| : 1
\leq k \leq n \}$. Hence
$$
\int^{\rho}_{0}\exp \left(  T_{n}^{*}(x) \right) \,dx =
\int_0^\infty e^t |\{x \in [0, \rho ]: T_{n}^{*}(x)
> t \}| \, dt \leq 2 \int_0^\infty \exp \left( t - t^2 /2 \|\langle
T\rangle^{2}_{n}\|_{\infty} \right) dt
$$
We deduce that
$$
\int^{\rho}_{0}\exp \left( T_{n}^{*}(x) \right)\,dx \le 2 \sqrt{2
\pi} \|\langle T\rangle_{n}\|_{\infty} \exp \left( \|\langle
T\rangle^{2}_{n}\|_{\infty} / 2  \right) ,\quad n=1,2,\dotsc
$$
Since $\|T_{n+1}-T_{n}\|_{\infty}\le 1+2^{-\beta}$ for any~$n$, one
has $\|\langle T\rangle^{2}_{n}\|_{\infty}\le n (1+2^{-\beta})^{2}$
for $n=1,2,\dotsc$ . We deduce that for any $\lambda >0$, one has
$$
\int^{\rho}_{0} \exp \left( \lambda T_{n}^{*}(x) \right) \,dx \le 2
(1+2^{-\beta}) \sqrt{2 \pi n}   \lambda \exp \left(
\frac{\lambda^{2}}{2} n(1+2^{-\beta})^{2} \right),\quad
n=1,2,\dotsc, .
$$
On the other hand, summation by parts gives that
$$
T_{n} = (1-2^{-\beta})\Gamma_{n-1}+ 2^{-n\beta}S_{n}.
$$
Hence
\begin{equation}\label{eq2.1}
\Gamma_n^{*} \le (1-2^{-\beta})^{-1} ( T_{n+1}^{*} +1)
\end{equation}
We deduce that for any $\quad n=1,2,\dotsc,$ and any $ \lambda>0$,
one has
$$
\int^{\rho}_{0} \exp \left( \lambda \Gamma_n^{*}(x) \right) \,dx \le
2 \frac{1+2^{-\beta}}{1- 2^{-\beta}} \sqrt{2 \pi (n+1) }   \lambda
\exp \left( \lambda (1-2^{-\beta})^{-1} \right) \exp \left(
\frac{1}{2}\left( \frac{1+2^{-\beta}}{1-2^{-\beta}}\right)^{2}
\lambda^{2} (n+1) \right) .
$$
Hence, the trivial estimate $\lambda (1-2^{-\beta})^{-1}\le
\lambda^{2}/2 +(1-2^{-\beta})^{-2}/2$ finishes the proof of (a).

The estimate~\eqref{eq2.1} gives
$$
\int^{\rho}_{0} |\Gamma_n^{*}(x)|^{2}\,dx \le 2 (1-2^{-\beta})^{-2}
\int^{\rho}_{0} |T^{*}_{n+1} (x)|^{2}\,dx+ 2 \rho
(1-2^{-\beta})^{-2}.
$$
Since by Doob's maximal inequality (\cite[p.493]{S})
$$
\int^{\rho}_{0}|T^{*}_{n+1}(x)|^{2}\,dx \le c \int^{\rho}_{0}
\langle T\rangle_{n+1}^{2}(x)\,dx \le c(n+1),
$$
(b) follows. Finally, the Law of the Iterated Logarithm applied to~$\{T_{n} \}$ gives
$$
\limsup_{n\to\infty} \frac{|T_{n}(x)|}{\sqrt{n \log\log n}} \le c\text{ a.e.\ }x.
$$
We deduce
$$
\limsup_{n\to\infty} \frac{|\Gamma_n(x)|}{\sqrt{n \log\log n}} \le
c(1-2^{-\beta})^{-1}\text{ a.e.\ }x
$$
which finishes the proof.
\end{proof}

\section{Averaging}\label{section3}

An averaging procedure due to J.~Garnett and P.~Jones (\cite{GJ})
will be used to go from the discrete situation of
Theorem~\ref{theo2} to the continuous one of Theorem~\ref{theo1}.

Given $x\in\mathbb{R}$, let $I_{k}^{\rho}(x)$ be the unique
interval of $\mathcal{D}(\rho)$ of length~$2^{-k}\rho$ which contains $x$. Given a function~$f\colon \mathbb{R}\to\mathbb{R}$ and an interval~$I=[a,b)$ we denote $\Delta f(I)=f(b)-f(a)$ and consider the dyadic martingale with respect to the filtration~$\mathcal{D}(\rho)$ given by
$$
S_{k}^{(\rho)}(f)(x)= \frac{\Delta f(I_{k}^{(\rho)}(x))}{2^{-k}
\rho},\quad k=0,1,2,\dotsc\,.
$$
If $f\in\Lambda_{\alpha}(\mathbb{R})$, we have
$\|S_{k}^{(\rho)}(f)\|_{\infty} \le (2^{k}/ \rho )^{\beta} \| f
\|_\alpha$, $k=0,1,\dotsc$ where $\beta=1-\alpha$. As in
Section~\ref{section2}, consider
\begin{equation}\label{eq3.1}
\Gamma_{n}^{(\rho)} (f)(x)= \Gamma_{n}^{(\rho)} (\{
S_{k}^{(\rho)}\})(x)=\sum^{n}_{k=1} 2^{-k\beta}\rho^{\beta}
S_{k}^{(\rho)}(f)(x) = \sum^{n}_{k=1} \frac{\Delta f
(I_{k}^{(\rho)}(x))}{(2^{-k}\rho)^{\alpha}}.
\end{equation}
The main purpose of this section is to describe an averaging argument with respect both $\rho\in [1,2]$ and translates of the dyadic net~$\mathcal{D}(\rho)$. We start with a preliminary result.

\begin{lemma}\label{lem3.1}
Let $f\colon \mathbb{R}\to \mathbb{R}$ be a locally integrable function. For $s\in\mathbb{R}$ let $f_{s}$ be the function defined by $f_{s}(x)=f(x-s)$, $x\in\mathbb{R}$. Then for any $x\in\mathbb{R}$ and any $k=1,2,\dotsc$, one has
$$
\int^{\rho}_{0}\Delta f_{s} (I_{k}^{(\rho)} (x+s))\,ds= 2^{k} \int_{0}^{2^{-k}\rho} (f(x+t)-f(x-t))\,dt.
$$
\end{lemma}

\begin{proof}
Fix $x\in\mathbb{R}$ and $k=1,2,\dotsc$ . Let $I_{k}^{(\rho)}(x)=[a,b)$. Fix an integer~$j$ with $0\le j\le 2^{k}-1$ and consider $[2^{-k}j\rho,2^{-k}(j+1)\rho)=J\cup K$ where $J=J(x)= [2^{-k}j\rho, 2^{-k}j\rho+b-x)$ and $K=K(x)=[2^{-k}j\rho+b-x, 2^{-k}(j+1)\rho)$. Note that for $s\in J$ one has $I_{k}^{(\rho)}(x+s)=[a+2^{-k}j\rho,b+2^{-k}j\rho)$ and
$$
\int_{J} \Delta f_{s} (I_{k}^{(\rho)}(x+s))\,ds = \int_{J}
(f(b+2^{-k}j\rho-s)-f(a+2^{-k}j\rho-s)) \,ds =
 \int_{0}^{b-x} (f(x+t) -f(x+t-2^{-k} \rho))\,dt.
$$
For $s\in K$ one has $I_{k}^{(\rho)} (x+s)=[a+2^{-k}(j+1)\rho,
b+2^{-k}(j+1)\rho)$ and
\begin{equation*}
\begin{split}
\int_{K}\Delta f_{s} (I_{k}^{(\rho)} (x+s))\,ds&= \int_{K} (f(b+2^{-k}(j+1)\rho -s)-f (a+2^{-k}(j+1)\rho -s))\,ds\\*[5pt]
&=\int_{b-x}^{2^{-k}\rho}(f(x+t)-f(x+t-2^{-k}\rho))\,dt.
\end{split}
\end{equation*}
Thus
$$
\int^{2^{-k}(j+1)\rho}_{2^{-k}j\rho}\Delta f_{s}(I_{k}^{(\rho)}(x+s))\,ds=\int_{0}^{2^{-k}\rho}(f(x+t)-f(x+t-2^{-k}\rho))\,dt
=\int_{0}^{2^{-k}\rho}(f(x+t)-f(x-t))\,dt.
$$
Adding on $j=0,\dotsc,2^{k}-1$, one finishes the proof.
\end{proof}

We now state the main result of this section.

\begin{proposition}\label{prop3.2}
Fix $0 < \alpha \leq 1$. Let $f$ be a locally integrable function.
For $s\in\mathbb{R}$ let $f_{s}$ be the function defined by
$f_{s}(x)=f(x-s)$, $x\in\mathbb{R}$. For $n=1,2,\dotsc$, consider
$\Gamma_{n}^{(\rho)}(f_{s})$ as defined in~\eqref{eq3.1}. Then for
any $x\in\mathbb{R}$, one has
$$
\int^{2}_{1}\int^{\rho}_{0}\Gamma_{n}^{(\rho)} (f_{s}) (x+s)\,ds
\frac{d\rho}{\rho^{2}} =\frac{1}{1+\alpha}
\int^{1}_{2^{-n}}\frac{f(x+t)-f(x-t)}{t^{1+\alpha}}\,dt+ A_{n}
(f)(x)
$$
where
$$
 |A_{n} (f)(x)|\le c(\alpha) \int^{1}_{2^{-n}}|f(x+t)-f(x-t)|\,dt
+c(\alpha) 2^{n(1+\alpha)}\int_{0}^{2^{-n}}|f(x+t)-f(x-t)|\,dt.
$$
In particular if $f\in \Lambda_{\alpha} (\Bbb R)$, one has
$\sup_{n,x} |A_{n} (f)(x)| < C(\alpha ) \|f \|_\alpha $
\end{proposition}

\begin{proof}
For $k=1,2,\dotsc$, consider
$$
B_{k}=\int^{2}_{1}\int^{\rho}_{0} \frac{\Delta (f_{s})
(I_{k}^{(\rho)}(x+s))}{(2^{-k}\rho)^{\alpha}} \,ds
\frac{d\rho}{\rho^{2}}.
$$
Lemma~\ref{lem3.1} gives that
$$
B_{k}=\int^{2}_{1}2^{k}\int^{2^{-k}\rho}_{0} \frac{ f(x+t)-
f(x-t)}{(2^{-k}\rho)^{\alpha}} \,dt \frac{d\rho}{\rho^{2}}.
$$
The change of variables $h=2^{-k}\rho$ gives
$$
B_{k}=\int^{2^{-k+1}}_{2^{-k}}\frac{1}{h^{2+\alpha}} \int^{h}_{0}
(f(x+t)-f(x-t))\,dt\,dh.
$$
Adding on $k=1,\dotsc,n$, one deduces
$$
\int^{2}_{1}\int^{\rho}_{0} \Gamma_{n}^{(\rho)} (f_{s}) (x+s)\,ds
\frac{d\rho}{\rho^{2}}=\int^{1}_{2^{-n}} \frac{1}{h^{2+\alpha}}
\int^{h}_{0}(f(x+t)-f(x-t))\,dt\,dh.
$$
Applying Fubini's Theorem one deduces
\begin{equation*}
\begin{split}
\int^{2}_{1} \int^{\rho}_{0} \Gamma_{n}^{(\rho)} (f_{s}) (x+s)\,ds
\frac{d\rho}{\rho^{2}}&=\frac{1}{1+\alpha} \int^{1}_{2^{-n}}
\frac{f(x+t)-f(x-t)}{t^{1+\alpha}}\,dt\\*[7pt] &\quad
-\frac{1}{1+\alpha} \int^{1}_{2^{-n}} (f(x+t)-f(x-t))\,dt\\*[7pt]
&\quad+\frac{2^{n(1+\alpha)}-1}{1+\alpha} \int^{2^{-n}}_{0}
(f(x+t)-f(x-t))\,dt
\end{split}
\end{equation*}
which finishes the proof.
\end{proof}

\section{Continuous setting}\label{section4}

In this section, the results of the dyadic model of Section~\ref{section2} and the averaging procedure of Section~\ref{section3} will be used to prove Theorem~\ref{theo1}.

Given $f\in\Lambda_{\alpha} (\mathbb{R})$ and $0<\varepsilon <1$,
pick an integer~$N$ such that
 $2^{-N-1} \le\varepsilon <2^{-N}$. Observe that $|\Theta_{\varepsilon} (f)(x)-\Theta_{2^{-N}}(f)(x)|\le 2 \|f\|_{\alpha}$.
 Hence the estimates of $\Theta_{2^{-N}} (f)(x)$ can be easily transferred to
 $\Theta_{\varepsilon}(f)(x)$.
 The main technical step in proving the relevant subgaussian estimate of $\Theta_{2^{-N}}(f)(x)$ is stated in next result.

\begin{proposition}\label{prop4.1}
Let $f\in \Lambda_{\alpha} ([-1,2])$ with $\|f\|_{\alpha}\le 1$. For
$x\in [0,1]$ and $N=1,2,\dotsc$, consider
\begin{align*}
\Theta_{2^{-N}} (f)(x) &=\int^{1}_{2^{-N}} \frac{f(x+h)-f(x-h)}{h^{\alpha}} \frac{dh}{h},\\
\Theta_{2^{-N}}^{*}(f)(x) &=\sup_{k\le N} |\Theta_{2^{-k}}(f)(x)|.
\end{align*}
Then, there exists a constant $c(\alpha)>0$ such that for any $\lambda>0$ and any $N=1,2,\dotsc$, one has
$$
\int^{1}_{0} \exp \left( \lambda \Theta_{2^{-N}}^{*}(f)(x) \right)
\,dx\le c(\alpha) \exp \left( c(\alpha)\lambda^{2}N \right).
$$
\end{proposition}

\begin{proof}
Consider the set $A=\{ (\rho,s): 1\le \rho\le 2,\, 0\le s\le \rho
\}$ and the measure~$d\mu$ defined as
$$
\mu (E)= \int_{E \bigcap A } ds \frac{d\rho}{\rho^{2}} \, , E
\subset {\Bbb R}^2
$$
 For any
$k=1,2,\dotsc$, Proposition~\ref{prop3.2} gives that
$$
\Theta_{2^{-k}}(f)(x)= (1 + \alpha) \int_{A}\Gamma^{(\rho)}_{k}
(f_{s})(x+s) \,d\mu (\rho,s) +A_{k} (f) (x) \, .
$$
Moreover there exists a constant $C=C(\alpha)$ such that
$\sup\limits_{k,x} |A_{k}(x)| \le C$. Here is where the
normalization $\|f\|_{\alpha}\le 1$ is used. Hence if $k$ and $N$
are integers with $k\le N$ we deduce
$$
|\Theta_{2^{-k}}(f)(x)|\le (1 + \alpha) \int_{A} (
\Gamma^{(\rho)}_{N} )^{*} (f_{s}) (x+s) \,d\mu (\rho,s)+C.
$$
Here
$(\Gamma_{N}^{(\rho)})^*(f_{s})(x)=\sup\{|\Gamma_{k}^{(\rho)}(f_{s})(x)|:k\le
N\}$. Hence for any $N=1,2,\dotsc$, one has
$$
\Theta_{2^{-N}}^{*}(f)(x)\le (1+ \alpha) \int_{A}
(\Gamma^{(\rho)}_{N} )^* (f_{s}) (x+s) \,d\mu (\rho,s)+C.
$$
Now, Jensen's inequality and Fubini's Theorem give that
\begin{equation}\label{eq4.1}
\int^{1}_{0} \exp \left( \lambda \Theta_{2^{-N}}^{*}(f)(x) \right)
\,dx \le \exp \left( \lambda C \right) \int_{A}\int^{1}_{0} \exp
\left( \lambda (\alpha + 1) ( \Gamma_{N}^{(\rho)} )^* (f_{s}) (x+s)
\right) \,dx\,d\mu (\rho,s).
\end{equation}
Recall that $\Gamma_{N}^{(\rho)}(f_{s})$ is defined via the
formula~\eqref{eq3.1} from the martingale~$S_{k}^{(\rho)}(f_{s})$
which is given by $S_{k}^{(\rho)}(f_{s}) (x)=
(f_{s}(b)-f_{s}(a))/(b-a)$, where $x\in
I_{k}^{(\rho)}(x)=[a,b)\in\mathcal{D}(\rho)$. The normalization
$\|f\|_{\alpha}\le 1$ gives that there exists an absolute
constant~$c_{1}>0$ such that $|S_{0}^{(\rho)}(f_{s})| \le c_{1}$ for
any $(\rho,s)\in A$. Recall that if $\|f\|_{\alpha}\le 1$, the
martingale~$S_{k}^{(\rho)}$ satisfies $\|S_{k}^{(\rho)}\|_{\infty}
\le (2^{k} / \rho )^{1- \alpha}$. According to (a) of
Theorem~\ref{theo2}, there exists a constant~$c_{1}(\alpha)>0$ such
that
$$
\int^{1}_{0} \exp \left( \lambda (1+ \alpha) (\Gamma_{N}^{(\rho)}
)^* (f_{s})(x+s) \right) \,dx \le c_{1}(\alpha) \exp \left(
c_{1}(\alpha)(c_{1}\lambda+\lambda^{2}N) \right).
$$
The trivial estimate $2 \lambda\le \lambda^{2}+1$ shows that there
exists a constant~$c_{2}(\alpha)> c_{1}(\alpha)$ such that
$$
\int^{1}_{0} \exp \left( \lambda (1+ \alpha) (\Gamma_{N}^{(\rho)}
)^* (f_{s})(x+s) \right) \,dx \le c_{2}(\alpha)
e^{c_{2}(\alpha)\lambda^{2}N}.
$$
By \eqref{eq4.1} one deduces
$$
\int^{1}_{0} \exp \left( \lambda \Theta_{2^{-N}}^{*}(f)(x)\right)
\,dx \le c_{2}(\alpha) \exp \left( C \lambda \right) \exp \left(
c_{2}(\alpha)\lambda^{2}N \right).
$$
Again the trivial estimate $2 \lambda \le \lambda^{2}+1$ finishes
the proof.
\end{proof}

Now the subgaussian estimate follows easily.

\begin{corollary}\label{coro4.2}
Let $f\in\Lambda_{\alpha}([-1,2])$ with $\|f\|_{\alpha}\le 1$. Then
there exists a constant $c(\alpha)>0$ such that for any $N>0$ and
any $t>0$ one has
$$
|\{ x\in [0,1]: \Theta_{2^{-N}}^{*}(f)(x) > \sqrt{N} t\}| \le
c(\alpha) \exp \left(-t^{2}/c(\alpha) \right).
$$
\end{corollary}

\begin{proof}
Let $E=\{x\in [0,1]: \Theta_{2^{-N}}^{*}(f)(x)>\sqrt{N}t\}$.
Previous Proposition~\ref{prop4.1} and Chebyshev inequality gives
that for any $\lambda>0$ one has $ \exp \left( \lambda
\sqrt{N}t\right) |E|\le c(\alpha) \exp \left( c(\alpha)\lambda^{2}N
\right)  $, that is,
$$|E| \!\le\! c(\alpha) \!\exp
(c(\alpha)\lambda^{2\!}N\!-\!\lambda \sqrt{\!N}t) \, .
$$
We take
$\lambda\!=\!t/2c(\alpha)\!\sqrt{\!N}$ and deduce $|E|\! \le\!
c(\alpha)\!\exp (-t^{2\!\!}/4c(\alpha))\!$ which finishes the proof.
\end{proof}

We can now prove Theorem~\ref{theo1}.

\begin{proof}[Proof of Theorem \ref{theo1}]
In the proof of part~(a) we can assume that $I$ is the unit interval
and $\|f\|_{\alpha}=1$. Given $0<\varepsilon < 1/2$, pick an
integer~$N$ such that $2^{-N-1}\le \varepsilon < 2^{-N}$. Since
$|\Theta_{\varepsilon}(f)(x)-\Theta_{2^{-N}}(f)(x)|\le 2$,
Corollary~\ref{coro4.2} gives that
$$
|\{x\in [0,1] : |\Theta_{\varepsilon} (f)(x)| > \sqrt{N} t\}| \le c(\alpha) e^{-t^{2}/c(\alpha)},
$$
for any $t>0$ such that $t\sqrt{N} >2$. Now (a) follows easily from
$$
\int^{1}_{0}|\Theta_{\varepsilon} (f)(x)|^{2}\,dx= 2
\int^{\infty}_{0} \lambda |\{ x\in [0,1]
:|\Theta_{\varepsilon}(f)(x)|>\lambda \}| \,d\lambda.
$$
The Law of the Iterated Logarithm of part~(b) follows from the subgaussian estimate
 of Corollary~\ref{coro4.2} via an standard Borel-Cantelli argument. Consider the set~$A$ of points $x\in [0,1]$ for which
$$
\Theta_{2^{-N}}^{*} (f)(x) > 2 c\sqrt{ N \log \log N}
$$
for infinitely many $N\ge 0$. Here $c=c_0 (\alpha)$ is a constant
which will be chosen later. Let $N_m=2^m$. If
$\Theta^*_{2^{-N}}(f)(x)>2c\sqrt{N\log\log N}$ and $N_{m-1}<N\le
N_m$ then
\[
\Theta^*_{2^{-N_m}}(f)(x)\ge
\Theta^*_{2^{-N}}(f)(x)>2c\sqrt{N\log\log N}\ge c\sqrt{N_m\log\log
N_m}.\]
Thus $A\subset \cap_k\cup_{m\ge k} A_{N_m}$ where
\[A_{N_m}=\{x: \Theta^*_{2^{-N_m}}f(x)>c\sqrt{N_m\log\log N_m}\}.\]
Now Corollary 4.2 with $t=c\sqrt{\log\log N_m}=c(\log
(m\log2))^{1/2}$ gives $|A_{N_m}|\le c(\alpha)(m\log
2)^{-c^2/c(\alpha)}$ and for $c^2>c(\alpha)$ the Borel-Cantelli
lemma gives $|A|=0$. Thus for almost every $x\in [0,1]$ one has
$$
\limsup_{N\to\infty} \frac{|\Theta_{2^{-N}}^{*}(f)(x)|} {\sqrt{N\log
\log N}} \le 2c
$$
and the proof is completed.
\end{proof}

\section{Sharpness}\label{section5}
In this section the sharpness of our results is discussed.

\subsection{Sharpness of Theorem 1}\label{subsec5.1}
Both parts~(a) and~(b) in Theorem~\ref{theo1} as well as
Proposition~\ref{prop4.1} and its Corollary~\ref{coro4.2}, are sharp
up to the value of the constants~$c(\alpha)$. Since
Theorem~\ref{theo1} follows from Corollary~\ref{coro4.2}, it is
sufficient to construct a
function~$f\in\Lambda_{\alpha}(\mathbb{R})$ for which
\begin{equation}\label{eq5.1}
\int^{1}_{0} |\Theta_{\varepsilon}(f)(x)|^{2}\,dx \ge c(\log (
1/\varepsilon )),\quad 0<\varepsilon <1/2
\end{equation}
and
\begin{equation}\label{eq5.2}
\limsup_{\varepsilon\to 0}
\frac{\Theta_{\varepsilon}(f)(x)}{\sqrt{\log ( 1/\varepsilon
)\log\log\log ( 1/\varepsilon )}} >c, \quad \text{a.e.\ $x\in
[0,1]$}
\end{equation}
for a certain constant $c=c(\alpha)>0$. Fix $0<\alpha<1$. As it is
usual in this kind of questions, the function~$f$ will be given by a
lacunary series. More concretely, consider
$$
f(x)=\sum^{\infty}_{j=0} 2^{-j\alpha} \sin (2 \pi 2^{j}x).
$$
Then,
\begin{equation*}
\begin{split}
\Theta_{2^{-N}}(f)(x)&= \int^{1}_{2^{-N}}
\frac{f(x+h)-f(x-h)}{h^{\alpha}}\frac{dh}{h}\\*[7pt]
&=2\sum^{\infty}_{j=0} 2^{-j\alpha}\left(
\int^{1}_{2^{-N}}\frac{\sin (2^{j} 2 \pi
h)}{h^{\alpha+1}}\,dh\right) \cos (2^{j} 2 \pi x)\\*[7pt]
&=2\sum^{\infty}_{j=0} c_{j,N}\cos (2^{j} 2 \pi x),
\end{split}
\end{equation*}
where
$$
c_{j,N}=\int^{2^{j}}_{2^{j-N}}\frac{\sin( 2 \pi
t)}{t^{\alpha+1}}\,dt.
$$
Integrating by parts one shows that there exists a constant~$c_{1}(\alpha)>0$ such that
$$
|c_{j,N}|\le c_{1}(\alpha) 2^{-(j-N)(\alpha+1)}, \quad j=1,2,\dotsc,
\quad N=1,2,\dotsc\, .
$$
Hence
\begin{equation}\label{eq5.3}
\sum_{j\ge N} |c_{j,N}| \le 2c_{1}(\alpha).
\end{equation}
On the other hand, using the estimate $|\sin t|\le t$, we have
\begin{equation}\label{eq5.4}
\sum^{N}_{j=0} \left| \int_{0}^{2^{j-N}}\frac{\sin 2 \pi
t}{t^{\alpha+1}}\,dt \right| \le\frac{2 \pi }{1-\alpha}
\sum^{N}_{j=0} 2^{(j-N)(1-\alpha)}\le c_{2}(\alpha).
\end{equation}
Using \eqref{eq5.3} and \eqref{eq5.4} one deduces that
\begin{equation}\label{eq5.5}
\Theta_{2^{-N}}(f)(x)=\sum^{N}_{j=0} b_{j}\cos (2^{j} 2 \pi
x)+E_{N}(x),
\end{equation}
where $|E_{N}(x)|\le c_{3}(\alpha)=2c_{1}(\alpha)+c_{2}(\alpha)$ for any $x\in\mathbb{R}$ and any $N=1,2,\dotsc$ and
$$
b_{j}=2\int^{2^{j}}_{0} \frac{\sin 2 \pi t}{t^{\alpha+1}}\,dt.
$$
Consider
$$A(\alpha)=\lim\limits_{j\to\infty} b_{j} = 2 \int _0^\infty \frac{\sin 2 \pi t}{t^{\alpha+1}}\,dt$$
and observe that $A(\alpha) > 0$. By orthogonality for $N$
sufficiently large one has
$$
\|\Theta_{2^{-N}}(f)\|^{2}_{L^{2}[0,1]} \ge \frac{1}{2}
A(\alpha)^{2}N
$$
which gives \eqref{eq5.1}.

A classical result by M.~Weiss (\cite{W}) gives that
$$
\limsup_{N\to\infty}\frac{\left|\sum\limits_{j=0}^{N}b_{j}\cos
(2^{j} 2 \pi x)\right|}{\sqrt{N\log\log N}}=A(\alpha).
$$
Thus, from \eqref{eq5.5} one deduces
$$
\limsup_{N\to\infty}\frac{|\Theta_{2^{-N}}(f)(x)|}{\sqrt{N\log\log
N}}=A(\alpha)
$$
which gives \eqref{eq5.2}.

\subsection{Cancellation}\label{subsec5.2}
Theorem~\ref{theo1} says that the uniform
estimate~$\|\Theta_{\varepsilon}(f)\|_{\infty}\le c(\log
1/\varepsilon)\|f\|_{\alpha}$, $0<\varepsilon <1/2$, can be
substantially improved at almost every point. This is due to certain
cancellations which occur in the integral defining
$\Theta_{\varepsilon}(f)(x)$. Actually there exist $f\in
\Lambda_{\alpha}(\mathbb{R})$ and $c=c(f)>0$ such that for any
$0<\varepsilon <1/2$ one has
\begin{equation}\label{eq5.6}
\int^{1}_{\varepsilon} \frac{|f(x+h)-f(x-h)|} {h^{\alpha}}
\frac{dh}{h} \ge c\log ( 1/\varepsilon )
\end{equation}
for almost every $x\in\mathbb{R}$. Let $b>1$ be a large positive integer to be fixed later. Consider
$$
f(x)=\sum^{\infty}_{j=0} b^{-j\alpha} \cos (b^{j}x), \quad x\in\mathbb{R}.
$$
Fix $k\ge 0$ and $h$ such that $b^{-k}/2\le h\le 2b^{-k}$. Observe that
$$
\frac{2}{h^{\alpha}} \sum_{j>k} b^{-j\alpha} \le c(\alpha)
b^{-\alpha}
$$
and
$$
\frac{1}{h^{\alpha}} \sum_{j<k} b^{-j\alpha} |\cos (b^{j}x+b^{j}h)-\cos (b^{j}x-b^{j}h)|\le c(\alpha) b^{\alpha-1}.
$$
On the other hand
$$
\cos (b^{k}x+b^{k}h)-\cos (b^{k}x-b^{k}h)= -2\sin (b^{k}x) \sin (b^{k}h).
$$
Hence
$$
\int^{2b^{-k}}_{b^{-k}/2} \frac{|f(x+h)-f(x-h)|}{h^{\alpha}} \frac{dh}{h}\ge c|\sin (b^{k}x)|-c(\alpha,b)
$$
where $c(\alpha,b)=c(\alpha) (b^{-\alpha}+b^{\alpha-1})$ and $c>0$.
Thus, if $b$ is taken sufficiently large so that $c(\alpha,b)<c/4$,
one has
$$
\int^{1}_{\varepsilon} \frac{|f(x+h)-f(x-h)|}{h^{\alpha}}
\frac{dh}{h}  > c t(\varepsilon,x)/4
$$
where $t(\varepsilon,x)$ is the  number of positive integers~$k$
such that $b^{-k}\ge 2\varepsilon$ which satisfy $|\sin (b^{k}x)|\ge
1/2$. The uniform distribution of $\{b^k x\}$ (see Corollary 4.3 of
\cite{KN}) gives that there exists a constant $c_1
>0$ such that $t(\varepsilon,x) \ge c_1 \ln (2\varepsilon)^{-1}/\ln
b$ almost every~$x \in \Bbb R$. So
 \eqref{eq5.6}~follows.

\section{Higher Dimensions}\label{section6}

Theorem~\ref{theo1} can be easily extended to higher dimensions. For
$0<\alpha<1$, let $\Lambda_{\alpha} (\mathbb{R}^{d})$ be the class
of functions~$f\colon \mathbb{R}^{d}\to\mathbb{R}$ for which there
exists a constant $c=c(f)>0$ such that $|f(x)-f(y)|\le
c\|x-y\|^{\alpha}$ for any $x,y\in\mathbb{R}^{d}$. The infimum of
the constants~$c>0$ verifying this estimate is denoted
by~$\|f\|_{\alpha}$. Lebesgue measure in $\mathbb{R}^{d}$ is denoted
by~$dm$. Next result is the higher dimensional analogue of
Theorem~\ref{theo1}

\begin{theo}\label{theo6.1}
Let $0<\alpha<1$ and $f\in\Lambda_{\alpha}(\mathbb{R}^{d})$. For
$0<\varepsilon<1/2$, consider
$$
\Theta_{\varepsilon}(f)(x)=\int_{\{\varepsilon \le\|h\|\le 1\}} \frac{f(x+h)-f(x-h)}{\|h\|^{\alpha}}�\frac{dm(h)}{\|h\|^{d}}.
$$
Then, there exists a constant~$c(\alpha , d)>0$ such that
\begin{enumerate}
\item[(a)] For any cube~$Q\subset \mathbb{R}^{d}$ with $m(Q)=1$, one has
$$
\int_{Q} |\Theta_{\varepsilon}(f)(x)|^{2}\le c(\alpha , d) (\log
1/\varepsilon) \|f\|^{2}_{\alpha}.
$$
\item[(b)] At almost every $x\in\mathbb{R}^{d}$, one has
$$
\limsup_{\varepsilon\to 0}
\frac{|\Theta_{\varepsilon}(f)(x)|}{\sqrt{\log ( 1/\varepsilon )
\log \log \log ( 1/\varepsilon )}} \le c(\alpha , d) \|f\|_{\alpha}.
$$
\end{enumerate}
\end{theo}

\begin{proof}
For any $\xi\in\mathbb{R}^{d}$, $|\xi|=1$, consider
$$
\Theta_{\varepsilon,\xi}(f)(x)=\int^{1}_{\varepsilon} \frac{f(x+\rho
\xi)-f(x-\rho\xi)}{\rho^{\alpha}}\frac{d\rho}{\rho}, \quad
x\in\mathbb{R}^{d}.
$$
Let $d\sigma(\xi)$ be the surface measure in the
sphere~$\{\xi\in\mathbb{R}^{d}: |\xi|=1\}$. Then
$$
\Theta_{\varepsilon}(f)(x)=\int_{\{|\xi|=1\}}
\Theta_{\varepsilon,\xi} (f)(x)\,d\sigma (\xi), \quad
x\in\mathbb{R}^{d}.
$$
We will take $\varepsilon=2^{-N}$ and will write $\Theta_{N,\xi}$ and $\Theta_{N}$ instead of $\Theta_{2^{-N},\varepsilon}$ and $\Theta_{2^{-N}}$. Also $\Theta^{*}_{N,\varepsilon}$, $\Theta^{*}_{N}$ will denote the maximal functions defined as
\begin{align*}
\Theta^{*}_{N,\xi}(f)(x)&=\sup \{ |\Theta_{k,\xi}(f)(x)|: k\le N\},\\
\Theta^{*}_{N}(f)(x)&=\sup \{ |\Theta_{k}(f)(x)|: k\le N\}.
\end{align*}
Then
$$
\Theta^{*}_{N}(f)(x) \le \int_{\{|\xi|=1\}} \Theta^{*}_{N,\xi}(f)(x) \,d\sigma(\xi), \quad x\in \mathbb{R}^{d},\quad N=1,2\dotsc\,.
$$
To prove (a) we  can assume that $Q$ is the unit cube and
$\|f\|_{\alpha}\le 1$. Jensen's inequality and Fubini's Theorem give
$$
\int_{Q} \exp \left( \lambda \Theta_{N}^{*}(f)(x)\right) \,dm(x)\le
\int_{\{|\xi|=1\}} \int_{Q} \exp \left( \lambda
\Theta^{*}_{N,\xi}(f)(x) \right) \,dm(x)\,d\sigma (\xi).
$$
Fixed $\xi\in \mathbb{R}^{d}$ with $|\xi|=1$, the inner integral can be understood as a $(d-1)$ dimensional integral of a one dimensional one to which we can apply Proposition~\ref{prop4.1}. Hence fixed $\xi\in\mathbb{R}^{d}$, $|\xi|=1$, we obtain
$$
\int_{Q} \exp \left(\lambda \Theta^{*}_{N,\xi}(f)(x) \right)
\,dm(x)\le c(\alpha) \exp \left( c(\alpha)\lambda^{2}N \right),
$$
for any $\lambda>0$ and any $N=1,2,\dotsc$ . We deduce
$$
\int_{Q} \exp \left( \lambda \Theta^{*}_{N}(f)(x) \right) \,dm(x)\le
c(\alpha)|\sigma(\{|\xi|=1\})| e^{c(\alpha)\lambda^{2}N},
$$
Now arguing as in Corollary~\ref{coro4.2}, one deduces the subgaussian estimate
$$
m\{x\in Q: |\Theta^{*}_{N}(f)(x)|> \sqrt{N} t\} \le c(\alpha , d)
\exp \left( -t^{2}/c(\alpha, d) \right).
$$
Arguing as in the proof of Theorem~\ref{theo1} one finishes the proof.
\end{proof}


$\quad$

$\quad$

 Jos\'e Gonz\'alez Llorente and Artur Nicolau

Departament de Matem{\`a}tiques

Universitat Aut{\`o}noma de Barcelona

08193 Bellaterra

Spain

jgllorente@mat.uab.cat

artur@mat.uab.cat



\begin{thebibliography}{CWW}
\bibitem[AP]{AP}
{\sc J. M.~Anderson and L.~D.~Pitt},
 Probabilistic behaviour of functions in the Zygmund spaces,
\emph{. Proc. London Math. Soc.} \textbf{59}, no~3 (1989)
558{\Ndash}592.

\bibitem[BM]{BM}
{\sc R.~Ba\~nuelos and C.~N.~Moore},
\emph{``Probabilistic behavior of harmonic functions''},
Progress in Mathematics~\textbf{175}, Birkh\"auser Verlag, Basel, 1999.

\bibitem[CWW]{CWW}
{\sc S.-Y.~Chang, J.~M.~Wilson, and T.~H.~Wolff}, Some weighted norm
inequalities concerning the Schr\"odinger operators, \emph{Comment.\
Math.\ Helv.} \textbf{60}, no.~2 (1985), 217{\Ndash}246.

\bibitem[E]{E}
{\sc K.~S.~Eikrem}, Hadamard gap series in growth spaces,
\emph{Collectanea Mathematica} \textbf{64}, no.~1 (2013),
1{\Ndash}15.

\bibitem[EM]{EM}
{\sc K.~S.~Eikrem and E.~Malinnikova},
Radial growth of harmonic functions in the unit ball,
 \emph{Math. \ Scand.} \textbf{110}, no.~1 (2012), 273{\Ndash}296

\bibitem[EMM]{EMM}
{\sc K.~S.~Eikrem and E.~Malinnikova, and P.~A.~Mozolyako},
Wavelet decomposition of harmonic functions in growth spaces,
arXiv:1203.5290v1 [math.FA].

\bibitem[GJ]{GJ}
{\sc J.~B.~Garnett and P.~W. Jones},
BMO from dyadic BMO,
\emph{Pacific J.\ Math.} \textbf{99}, no.~2 (1982), 351{\Ndash}371.

\bibitem[H]{H}
{\sc G.~H.~Hardy},  Weierstrass's non-differentiable function,
\emph{Trans.\ Amer.\ Math.\ Soc.\ } \textbf{17}, no.~3 (1916),
301{\Ndash}325.

\bibitem[KN]{KN}
{\sc L.~Kuipers and H.~Niederreiter} ,
 Uniform distribution of sequences,
  Pure and Applied Mathematics. Wiley-Interscience John
Wiley and Sons, New York-London-Sydney, 1974.

\bibitem[LM]{LM}
{\sc Yu.~Lyubarskii and E.~Malinnikova},
Radial oscillation of harmonic functions in the Korenblum class,
\emph{Bull.\ London Math.\ Soc.} \textbf{44}, no.~1 (2012), 68{\Ndash}84.

\bibitem[S]{S}
{\sc A.N.~Shiryaev}, \emph{``Probability''},  Graduate Texts in
Mathematics,~\textbf{95}, Springer-Verlag, 1996.

\bibitem[SV]{SV}
{\sc L.~Slavin and A.~Volberg},
The $s$-function and the exponential integral, in: \emph{``Topics in Harmonic Analysis and Ergodic Theory''},
Contemp.\ Math.~\textbf{444}, Amer.\ Math.\ Soc., Providence, RI, 2007, pp.,~215{\Ndash}228.

\bibitem[W]{W}
{\sc M.~Weiss},
The law of the iterated logarithm for lacunary trigonometric series,
\emph{Trans.\ Amer.\ Math.\ Soc.}~\textbf{91} (1959), 444{\Ndash}469.



\end{thebibliography}
\end{document}